\documentclass[11pt,twoside]{amsart}
\usepackage{amsmath}
\usepackage{amsthm}
\usepackage{amsfonts}
\usepackage{amssymb}
\usepackage[all]{xy}
\usepackage{alltt}
\usepackage{enumitem}
\usepackage{stmaryrd}
\usepackage{comment}
\usepackage{xspace}
\usepackage{comment}

\theoremstyle{plain}
\newtheorem{prop}{Proposition}[section]

\newtheorem{lem}{Lemma}
\newtheorem*{thmA}{Theorem}
\newtheorem*{corollary}{Corollary}
\newtheorem*{P69}{{Problem~69}}
\newtheorem*{P70}{Problem~70}
\newtheorem*{P72}{Problem~72}

\newtheorem{thm}[prop]{Theorem}
\theoremstyle{definition}

\newtheorem{remark}[prop]{Remark}

\newcommand{\Z}{\ensuremath{\mathbb{Z}}}

\newcommand{\R}{\ensuremath{\mathbb{R}}}

\newcommand{\n}{\noindent}

\usepackage[margin=1.25in]{geometry}

\usepackage[
        colorlinks=true,urlcolor=blue,linkcolor=blue,citecolor=blue
]{hyperref}

\begin{document}  
 
% Article information
\title{Fuchs' problem for indecomposable abelian groups}
\date{\today}
 
% Author information
\author{Sunil Chebolu}
\address{Department of Mathematics \\
Illinois State University \\
Normal, IL 61790, USA}
\email{schebol@ilstu.edu}

\author{Keir Lockridge} 
\address {Department of Mathematics \\
Gettysburg College \\
Gettysburg, PA 17325, USA}
\email{klockrid@gettysburg.edu}

\thanks{
The first author is supported by an NSA grant (H98230-13-1-0238). }

% AMS information
 
% Abstract
\begin{abstract}
More than 50 years ago, L\'{a}szl\'{o} Fuchs asked which abelian groups can be the group of units of a commutative ring. Though progress has been made, the question remains open. We provide an answer to this question in the case of indecomposable abelian groups by classifying the indecomposable abelian groups that are realizable as the group of units of a ring of any given characteristic.
\end{abstract}
 
\maketitle
\thispagestyle{empty}

\tableofcontents

% Body of the paper

\section{Introduction}\label{sec:introduction}

In 1960, L\'{a}szl\'{o} Fuchs posed the following problem: 

\begin{P72}[\cite{Fuchs}] Characterize the groups which are the groups of all units in a commutative and associative ring with identity. \end{P72}

\n This problem remains open, though modified versions have been studied extensively in the intervening years. One way to simplify the problem is to narrow the class of groups or rings involved. For example, the problem was solved for cyclic groups in \cite{PS}, and in \cite{Dolzan}, the class of finite rings is considered. Restricting the class of rings to fields, Fuchs also posed:

\begin{P69}[\cite{Fuchs}] Find necessary and sufficient conditions for a group to be the multiplicative group of a field. \end{P69}

\n This problem also remains unsolved, though much progress has been made (see \cite{CMN} for a survey of results in this area). In the next question, Fuchs asks whether the torsion subgroup of the multiplicative group of a field is necessarily a direct summand.

\begin{P70}[\cite{Fuchs}] Does every multiplicative group of fields split (as a mixed group)? \end{P70}

\n This question was answered negatively by Cohn in \cite{Cohn}. In \cite{cl-1}, the authors investigated a combination of the previous two problems by asking the question: when is the multiplicative group of a field indecomposable? (A group $G$ is indecomposable if it is not isomorphic to the direct product of two non-trivial groups.) A nearly complete answer is obtained, though it is still unknown whether there exists an infinite field of characteristic 2 whose multiplicative group is indecomposable. In the present work, we solve Problem 72 for indecomposable abelian groups. That is, we classify the indecomposable abelian groups which are realizable as the multiplicative group of a ring, and our classification is stratified by characteristic. We work with unital rings in this paper; the answer we obtain is the same for both the class of rings and the class of commutative rings.

The next theorem summarizes our results. We outline the proof immediately, which serves as a brief guide to the remaining sections. We consider non-torsion-free groups in \S \ref{sec:torsion} and torsion-free groups in \S \ref{sec:torsion-free}. Our paper is relatively self-contained. We use Kulikov's classification (\cite{K}) of indecomposable abelian groups which are not torsion-free (see \cite{Robinson} for a proof of this result in English), and we use the fact that every torsion-free abelian group admits a linear order (see \cite{Levi}). We do not use the hypothesis that the group of units is indecomposable in the torsion-free case. Otherwise, we use only basic facts from ring theory (see \cite{DummitFoote}). Let $C_r$ denote the (multiplicative) cyclic group of order $r$, and let $\displaystyle C_{p^{\infty}}  = \bigcup_{n\ge 0} C_{p^n}$ denote the quasi-cyclic $p$-group.

\begin{thmA}
Let $R$ be a unital ring. If the group of units $R^\times$ is an indecomposable abelian group, then $\mathrm{char}\, R = 0, 2, 4, q,$ or $2q$, where $q$ is a Fermat prime. The following is a complete list of the indecomposable abelian groups which are realizable as the group of units for a ring $R$ with $c = \mathrm{char} R$. In the third column, we give an example of such a ring in each case.

\begin{center}
\begin{tabular}{| l | l | l |}
\hline
$c = 0$ & $C_2$ & $\Z$ \\ 
& $C_4$ & $\Z[i]$ \\ \hline
$c = 2$ & $C_1$ & $\mathbf{F}_2$ \\
& $C_2$ & $\mathbf{F}_2[x]/(x^2)$ \\
& $C_4$ & $\mathbf{F}_2[x]/(x^3)$ \\
& $C_p$, $p$ a Mersenne prime & $\mathbf{F}_{p+1}$\\
& $G$ any indecomposable torsion-free abelian group & $\mathbf{F}_2[G]$\\ \hline
$c = 4$ &$C_2$ & $\Z_4$\\
& $C_4$ & $\Z_4[x]/(x^2 - 2, 2x)$ \\ \hline
$c = 3$ & $C_2$ & $\mathbf{F}_3$ \\
& $C_8$ & $\mathbf{F}_9$ \\ \hline
$c = 6$ & $C_2$ & $\mathbf{F}_3 \times \mathbf{F}_2$ \\
& $C_8$ & $\mathbf{F}_9 \times \mathbf{F}_2$ \\ \hline
$c = q$ & $C_{q-1}$, $q > 3$ a Fermat prime & $\mathbf{F}_q$ \\ \hline
$c = 2q$ & $C_{q-1}$, $q > 3$ a Fermat prime & $\mathbf{F}_q \times \mathbf{F}_2$ \\\hline
\end{tabular}
\end{center}

\n In particular, there is no ring $R$ whose unit group is a quasi-cyclic $p$-group $C_{p^\infty}$.
\end{thmA}
\begin{proof} To prove this theorem, we consider the torsion-free and non-torsion-free cases separately. If $G$ is a torsion-free abelian group, then $(\mathbf{F}_2[G])^\times \cong G$ by Theorem \ref{torsion-free} in \S \ref{sec:torsion-free} (this uses the fact that every torsion-free group admits a linear ordering). Since $-1 \in R^\times$ has order 2 unless $\mathrm{char}\, R = 2$, no other characteristic need be considered in this case. On the other hand, if an indecomposable abelian group $G$ has a non-trivial torsion element, then for some prime $p$, $G$ is either a cyclic $p$-group $C_{p^n}$ or a quasi-cyclic $p$-group $C_{p^\infty}$ (see \S \ref{sec:torsion} for more details). In Proposition \ref{char}, we narrow the possible values of $c = \mathrm{char}\, R$ to the list given in the statement of the theorem, and for each such value we determine the indecomposable abelian groups which are realizable as the group of units of a ring in Propositions \ref{char0} ($c = 0$), \ref{char2} ($c = 2$), \ref{char4} ($c = 4$), \ref{charq} ($c = q$), and \ref{char2q} ($c = 2q$).\end{proof}

The peculiarity of the prime 2 is evident in this theorem. It is conjectured that there are infinitely many Mersenne primes.  If this conjecture is true, then the characteristic 2 case is the only situation where there are infinitely many non-isomorphic examples, in both the torsion-free and non-torsion-free cases.  (There are infinitely many indecomposable torsion-free abelian groups; consider, for example, the groups $\mathbb{Z}[1/p]$ for $p$ a prime.) In stark contrast to the non-torsion-free case, the class of indecomposable torsion-free abelian groups is quite complicated and not well understood. The following observations, which follow from the above theorem, relate Mersenne primes, Fermat primes, and Fuchs' problem.

\begin{enumerate}
\item For $p$ an odd prime, $p$ is a Mersenne prime if and only if $C_p$ is the group of units of a ring of characteristic 2.
\item For $p$ an odd prime, $p$ is a Fermat prime if and only if there exists an indecomposable group that is group of units of a ring of characteristic $p$.
\item There are infinitely many Mersenne primes if and only if there is an infinite family of non-isomorphic indecomposable finite abelian groups which are realizable as the group of units in a ring of characteristic 2.
\end{enumerate}

\n In \cite{CLY}, we found several other characterizations of Mersenne primes.
\begin{comment}
We record the connection between the number of Mersenne primes and Fuchs' problem in the following corollary.

\begin{corollary} There are infinitely many Mersenne primes if and only if there is an infinite family of non-isomorphic indecomposable finite abelian groups which are realizable as the group of units in a ring of characteristic 2. \end{corollary}
\end{comment}

The smallest Fermat prime 3 also occupies an unusual place in the classification: the cyclic group $C_8$ appears as the multiplicative group of a ring of characteristic 3 (and 6), but it is not the multiplicative group of a ring of characteristic $q$ (or $2q$) for any other Fermat prime $q$. This oddity is related to the only non-trivial solution in Catalan's Conjecture; see \cite{cl-1} for more details. Finally, it is noteworthy that the only cyclic groups of odd prime power order appearing in the theorem are the groups $C_p$ for $p$ a Mersenne prime.

The authors are grateful to the referee for providing detailed and helpful comments on this article.
\section{Groups with torsion} \label{sec:torsion}

In this section, we determine all indecomposable abelian groups which are not torsion-free and are realizable as the group of units of a ring. Kulikov (\cite{K}) proved that any indecomposable abelian group that is not torsion-free is isomorphic to either a cyclic or quasi-cyclic $p$-group for some prime $p$ (see also \cite[4.3.12]{Robinson}). Denote by $C_{p^n}$ the (multiplicative) cyclic group of order $p^n$, and let $C_{p^\infty}$ denote the unique (up to isomorphism) quasi-cyclic $p$-group. The latter group may be realized as the union of all cyclic groups of order a power of $p$:
\[ C_{p^{\infty}}  = \bigcup_{n\ge 0} C_{p^n}.\]

In the next proposition, we prove that a ring whose group of units is cyclic or quasi-cyclic must have characteristic $0, 2, 4, q$ or $2q$, where $q$ is a Fermat prime. We then determine, for each characteristic, the cyclic and quasi-cyclic groups that are realizable as the group of units of a ring. Our classification enlarges the scope of \cite{PS} in that quasi-cyclic groups are considered (and, in fact, eliminated). We also provide slightly more information by keeping track of the characteristics of the rings involved.

\begin{prop}  \label{char}
Let $R$ be a ring such that $R^{\times} \cong C_{p^{n}}$, where $1 \le n \le \infty$. The characteristic of $R$ is  $0, 2, 4, q$ or $2q$, where $q$ is a Fermat prime. Moreover, when the characteristic is not equal to $2$,  we have $p = 2$.
\end{prop}

\begin{proof}
Let $c$ denote the characteristic of $R$. If $c > 0$, then $\mathbb{Z}_{c}$ is a subring of $R$, and we have
\[ \mathbb{Z}_{c}^{\times} \hookrightarrow R^{\times} \cong C_{p^{n}}.  \]
Since every finite subgroup of $C_{p^{n}}$ is  cyclic, we conclude that 
$\mathbb{Z}_{c}^{\times}$ is cyclic. It is well known that $\mathbb{Z}_{c}^{\times}$  is cyclic if and only if $c = 2, 4$, $q^{r}$ or $2q^{r}$ where $q$ is an odd prime and $r$ is a positive integer. For the cases $c = q^r$ and $c = 2q^r$, note that finite subgroups of $C_{p^{n}}$ are cyclic $p$-groups, so
\[ (\mathbb{Z}_{2q^{r}})^{\times} \cong \mathbb{Z}_{q^{r}}^{\times} \cong C_{\phi(q^{r})} = C_{q^{r-1}(q-1)}\]
is a $p$-group. This implies $r=1$ and $q = p^k + 1$ for some positive integer $k$. This forces $p = 2$ and hence $q$ is a Fermat prime.

When $c \neq 2$, the element $-1 \in \R^\times$ has order 2. This forces $p = 2$, and the proof is complete.
\end{proof}

We now consider each of the possible characteristics separately. The characteristic 4 case will reduce to the characteristic 2 case, and the characteristic $2q$ case will reduce to the characteristic $q$ case. This leaves characteristics $0, 2$, and $q$; our analysis in each case is different. For the remainder of this section, $p$ will always be a prime number and $1 \leq n \leq \infty$ (i.e., $n$ is either a positive integer or $n = \infty$, corresponding to the quasi-cyclic case).

\begin{prop} \label{char0}
There is a ring $R$ of characteristic $0$ such that $R^{\times} \cong C_{p^{n}}$ if and only if $(p, n) = (2, 1)$ or $(p, n) = (2, 2)$.
\end{prop}
\begin{proof}
First observe that when $p = 2$ and $n  = 1$ or  $2$  we have $\mathbb{Z}^\times \cong C_2$ (generated by $-1$) and $\mathbb{Z}[i]^\times \cong C_4$ (generated by $i = \sqrt{-1}$).

For the converse, we have $p = 2$ by Proposition \ref{char}. For $3 \le n \le \infty$, we claim that there is no ring $R$ of characteristic $0$ whose unit group is isomorphic to $C_{2^n}$. Assume to the contrary that such a ring exists.  Since the characteristic of $R$ is $0$, $R$ is a $\mathbb{Z}$-algebra. Since $n \geq 3$, the ring $R$ contains an $8$th root of unity $\zeta_8$. Consider the natural ring homomorphism 
\[ f \colon  \mathbb{Z}[x]  \longrightarrow R \]
which evaluates a polynomial at $\zeta_{8}$. Since $\zeta_{8}^{4} = -1$, the above map factors as:
\[
\xymatrix{
\mathbb{Z}[x] \ar[rr]^{f}  \ar@{>>}[dr]_{\pi}&  & R \\  
& \frac{\mathbb{Z}[x]}{(x^{4}+ 1)}. \ar[ur] & 
}
\] 
We will now find two elements in the kernel of $f$ that together generate an ideal containing a non-zero integer. This will give rise to a contradiction since the $\mathbb{Z}$-algebra homomorphism $f$ must be the identity map on $\mathbb{Z}$.

The element $x^{4}+1$ is one obvious element in the kernel of $f$.  We claim that there exists an integer $k \geq 1$ such that $(1+x+x^{2})^{2^{k}} - 1$ is another element in the kernel of $f$.  Since all units in $R$ have order a power of 2 by assumption, to prove our claim it is enough to show that $f(1+x+x^{2})$ is a unit in $R$. Since the above triangle commutes and ring maps preserve units, it suffices to prove that $\pi(1+x+x^{2})$ is a unit in 
$\mathbb{Z}[x]/{(x^{4}+ 1)}$. This is indeed the case, as $1 - x^{2}+x^{3}$ is the inverse of $1 + x + x^{2}$ in this quotient ring: $$(1 - x^{2}+x^{3})(1 + x + x^{2}) = 1+x+x^5 = 1 +x(x^4+1) \equiv 1 \mod (x^4+1).$$

Now let $\beta(x) = x^4 + 1$ and let $\alpha(x) = (1+x+x^{2})^{2^{k}} - 1$, choosing $k$ as in the previous paragraph. To show that the ideal $I$ generated by $\alpha(x)$ and $\beta(x)$ in $\mathbb{Z}[x]$ contains a non-zero integer, it suffices to show that these two polynomials are relatively prime in $\mathbb{Q}[x]$. Since $\mathbb{Q}[x]$ is a Euclidean domain, we will then have two rational polynomials $a(x)$ and $b(x)$ such that $1 = a(x) \alpha(x) + b(x) \beta(x)$. Now for some positive integer $m$ we have $ma(x), mb(x) \in \Z[x]$, and hence \[ m = (m a(x)) \alpha(x) + (m b(x)) \beta(x) \] in $\Z[x]$.
This last equation shows that a non-zero integer belongs to $I$, as desired.  

It therefore remains to show that  $\alpha(x)$ and $\beta(x)$ are relatively prime in $\mathbb{Q}[x]$. To this end, 
first note  that $\beta(x) = x^{4}+ 1$ is an irreducible polynomial in $\mathbb{Q}[x]$. Hence we need only show that $\beta(x)$ does not divide $\alpha(x)$ in $\mathbb{Q}[x]$.  This is indeed not possible because it can easily be verified that the complex root $\zeta = e^{i \pi/4}$ of $\beta(x)$ is not a root of $\alpha(x)$ since $1 + \zeta + \zeta^2$ does not lie on the unit circle (so it is not a root of unity).\end{proof}

In analyzing rings of positive characteristic, we will make use of the following proposition.

\begin{prop} \label{rel-prime}
Let $q$ and $p$ be distinct primes and let $1 \leq n \leq \infty$. There exists a ring $R$ of characteristic $q$ with $R^{\times} \cong C_{p^{n}}$ if and only if
\begin{enumerate}
\item $n = 1$, $q = 2$, and $p$ is a Mersenne prime, or\label{mp}
\item $n = 3$, $q = 3$, and $p = 2$, or\label{oddball}
\item $n < \infty$, $q = 2^n + 1$ is a Fermat prime, and $p = 2$.\label{fp}
\end{enumerate}
\end{prop}
\begin{proof}
Let $R$ be ring with prime characteristic $q$ such that $R^\times \cong C_{p^n}$ for some $1 \leq n \leq \infty$. If $n < \infty$, let $k = n$; if $n = \infty$, choose $k > 1$ so that $p^k > q^2$. There exists a unit $\zeta_{p^{k}}$ of order $p^{k}$ in $R^\times$. We then have a natural surjective ring homomorphism
\[ \frac{\mathbf{F}_{q}[x]}{(x^{p^{k}} - 1)}  \longrightarrow \mathbf{F}_{q}[\zeta_{p^{k}}] \subseteq R \]
determined by evaluation of polynomials at $\zeta_{p^k}$.
Since $q$ and $p$ are distinct primes, the polynomial $x^{p^{k}} - 1$ and its formal derivative are relatively prime in $\mathbf{F}_{q}[x]$. Therefore $x^{p^{k}} - 1$ decomposes as 
a product of distinct irreducible polynomials in  $\mathbf{F}_{q}[x]$.  As a result,  $\frac{\mathbf{F}_{q}[x]}{(x^{p^{k}} - 1)} $ is a product of finite fields of characteristic $q$.  Since a quotient of a product of these fields is a again a product of such fields (possibly with fewer factors), the subring $ \mathbf{F}_{q}[\zeta_{p^{k}}]$ is also a product of finite fields of characteristic $q$:
\[  \mathbf{F}_{q}[\zeta_{p^{k}}] \cong \mathbf{F}_{q^{m_{1}}}  \times \cdots \times  \mathbf{F}_{q^{m_{t}}},\]
for some integers $m_{i}$.  Taking units, we have 
\[ \mathbf{F}_{q}[\zeta_{p^{k}}]^{\times} \cong  C_{q^{m_{1}}-1} \times \cdots \times C_{q^{m_{t}} - 1}. \]

Note that $ \mathbf{F}_{q}[\zeta_{p^{k}}]^{\times} $ is a finite subgroup of $C_{p^{n}}$ of order at least $p^{k}$ because $\zeta_{p^{k}}$ has order $p^{k}$. This implies that $ \mathbf{F}_{q}[\zeta_{p^{k}}]^{\times}  \cong C_{p^{r}}$ where $r \ge k$.  This is an indecomposable abelian group which is isomorphic to the product of cyclic groups above.  This implies that $C_{p^{r}} \cong C_{q^{m_{i}}  -1}$ for some $i$. Thus $p^{r} = q^{m_{i}} - 1$ and $n \geq r \ge k$.

To finish, we will need two lemmas from \cite{cl-1} which we now state for the convenience of the reader. We have changed the notation to match the notation in the present work.

\begin{lem}[{\cite[Proposition 2]{cl-1}}] Let $q$ be a prime. The quantity $q^{m} - 1$ is a power of $2$ if and only if $q$ is a Fermat prime and $m = 1$, or $q = 3$ and $m = 2$. \label{even}\end{lem} 

\begin{lem}[{\cite[Proposition 3]{cl-1}}] Let $q$ be a prime. Suppose $q^m - 1 = p^r$ for some odd prime $p$ and positive integer $r$. Then, $r = 1$, $q = 2$, and $p = 2^m - 1$ is a Mersenne prime. \label{odd}\end{lem}

\n First, suppose $n < \infty$, in which case $n = r = k$. If $p = 2$, then by Lemma \ref{even} either  $q$ is a Fermat prime and $m_i = 1$, in which case (\ref{fp}) holds, or $q = 3$ and $m_i = 2$, in which case (\ref{oddball}) holds. If $p$ is odd, then $q = 2$, and by Lemma \ref{odd} we have $r = n = 1$ and $p$ is a Mersenne prime, in which case (\ref{mp}) holds. Next, suppose $n = \infty$. If $p = 2$, then by Lemma \ref{even} $m_i \leq 2$ so $q^2 < p^k \leq p^r  = q^{m_i} - 1 \leq q^2 - 1$, a contradiction. If $p$ is odd, then, by Lemma \ref{odd}, $q = 2$ and $r = 1$. This forces $k = 1$, a contradiction.

It remains to show that in each of the three listed cases, there exists a ring $R$ with $R^\times \cong C_{p^n}$. For (\ref{mp}), let $p$ be a Mersenne prime, and take $R = \mathbf{F}_{p+1}$. For (\ref{oddball}), take $R = \mathbf{F}_9$. Finally, for (\ref{fp}), let $q = 2^n + 1$ be a Fermat prime and take $R = \mathbf{F}_q$.
\end{proof}

\begin{prop} \label{char2}
There is a ring $R$ of characteristic $2$ such that $R^{\times} \cong C_{p^{n}}$ if and only if $(p, n) = (2, 1)$, $(p, n) = (2, 2)$, or $p$ is a Mersenne prime and $n = 1$. 
\end{prop}

\begin{proof} By Proposition \ref{rel-prime}, it suffices to assume $p = 2$. As one can easily check, \[ (\mathbf{F}_2[x]/(x^2))^\times \cong C_2 \ \  \text{ and }\ \   (\mathbf{F}_2[x]/(x^3))^\times \cong C_4, \] so it remains to prove that there is no ring $R$ of characteristic 2 with $R^\times \cong C_{2^n}$ for $3 \leq n \leq \infty$.  Let $R$ be such a ring, and consider the subring $\mathbf{F}_2[\zeta_8]$ of $R$. We have a natural surjective ring homomorphism 
\[ \frac{\mathbf{F}_2[x]}{(x^8-1)} \longrightarrow  \mathbf{F}_2[\zeta_8] \]
determined by evaluation at $x = \zeta_8$. The group $(\mathbf{F}_2[\zeta_8])^\times$ is a subgroup of $ C_{2^{n}}$ with at least $8$ elements, and hence $(\mathbf{F}_2[\zeta_8] )^\times \cong C_{2^k}$, where $k \ge 3$. We claim, however, that there is no quotient of the ring ${\mathbf{F}_2[x]}/{(x^8-1)}$ whose unit group is cyclic of order at least 8. First note that since the characteristic of $R$ is 2, 
\[\frac{\mathbf{F}_2[x]}{(x^8-1)} \cong \frac{\mathbf{F}_2[x]}{(x-1)^8} \cong \frac{\mathbf{F}_2[x]}{(x^8)}.\]
Since $\mathbf{F}_2[x]$ is a principal ideal domain, the non-trivial quotients of this ring (up to isomorphism) are given by ${\mathbf{F}_2[x]}/{(x^t)}$ where $1 \le t \le 8$. A polynomial in any of these rings is a unit if and only if its constant term is 1. This implies that ${\mathbf{F}_2[x]}/{(x^t)}$  has exactly $2^{t-1}$ units. Since we need at least 8 units, $t \ge 4$. The possible candidates for our quotient are therefore the rings ${\mathbf{F}_2[x]}/{(x^t)}$ where $4 \le t \le 8$. However, none of these rings have cyclic multiplicative groups. To see why, first note that since $t \geq 4$, we have $2t -2 \geq 2t - 4 \geq t$. Hence:
\[ (1+x^{t-1})^2 = 1 + x^{2t-2} \equiv 1 \mod (x^t) \]
and 
\[ (1+x^{t-2})^2 = 1 + x^{2t-4} \equiv 1 \mod (x^t).\]
This proves that the distinct non-identity elements $1 + x^{t-1}$ and $1 + x^{t-2}$ each have order two, so the group of units cannot be cyclic.\end{proof}

\begin{prop} \label{char4}
There is a ring $R$ of characteristic $4$ such that $R^{\times} \cong C_{p^{n}}$ if and only if $(p, n) = (2, 1)$ or $(p, n) = (2, 2)$.
\end{prop}

\begin{proof}
One can check that $\Z_4$ and $\Z_4[x]/(2x, x^2 - 2)$ are rings of characteristic 4 with multiplicative groups $C_2$ and $C_4$, respectively.

For the converse, we know $p = 2$ by Proposition \ref{char}. Now suppose $R$ is a ring of characteristic 4 with $R^\times \cong C_{2^n}$ and $4 \leq n \leq \infty$. Consider the quotient map $R \rightarrow R/(2)$.  We claim that $(R/(2))^{\times}$ is either $C_{2^{n}}$ or $C_{2^{n-1}}$.  We first argue that this natural quotient map induces a surjective map on units. To this end, let $x$ be an element of $R$ that maps to a unit in the quotient. Then, there exists an element $y$ of $R$ such that $xy = 1 \mod (2)$ and hence $xy = 1 + 2t$ for some 
$t$ in $R$.  Now, squaring this equation we obtain that $(xy)^2 = 1$ (using the fact that the characteristic of the ring is 4).  Thus $xy$ is a unit of $R$, and so it must be the case that $xy = 1$ or $xy = -1$, because $R^\times$ has a unique subgroup of order 2.  This implies that $x$ has a right inverse in $R$, and one may similarly prove that $x$ has a left inverse in $R$. We therefore obtain that $x$ is a unit of $R$.

We next argue that the kernel of $R^{\times} \rightarrow  R/(2)^{\times}$ is a subgroup of $\langle -1 \rangle$.  Indeed, if $x$ is an element of the kernel, then $x = 1 + 2t$, and squaring gives $x^2 = 1$, so $x = 1$ or $-1$ because $-1$ is the unique element of order 2 in $R^\times$.  If the kernel is trivial, then $R^{\times} \cong  C_{2^{n}}$.   If the kernel is $\langle -1 \rangle$ and $n < \infty$, then 
\[ (R/(2) )^{\times}  \cong R^{\times}/\langle-1\rangle \cong R^{\times}/C_{2} \cong C_{2^{n}}/C_{2} = C_{2^{n-1}}.\]
If the kernel is $\langle -1 \rangle$ and $n = \infty$, then
\[(R/(2) )^{\times}  \cong R^{\times}/\langle -1\rangle \cong R^{\times}/C_{2} \cong C_{2^{\infty}}/C_{2} = \bigcup_{n \ge 2} C_{2^{n}}/C_{2} = \bigcup_{n \ge 2} C_{2^{n-1}} = C_{2^{\infty}}.\]
In both cases we have  shown that the multiplicative group of $R/(2)$  (a ring of characteristic $2$) is isomorphic to $C_{2^{k}}$ for $3 \le k \le \infty$. This contradicts the previous proposition.

The only remaining case is $p = 2$ and $n = 3$. Suppose to the contrary that there is a ring $R$ of characteristic $4$ such that $R^\times \cong C_8$. 
Let $\theta$ generate $R^\times$. Then we have a ring homomorphism
\[ \frac{\mathbb{Z}_4[x]}{(x^8 -1)} \longrightarrow R\]
which sends $x$ to $\theta$. The image of this map is a finite ring $S$ of characteristic $4$ whose unit group is $C_8$. Since $C_8$ is an indecomposable group, it follows that there is a primary finite ring of characteristic $4$ (some factor of $S$ in its primary decomposition) which has $C_8$ as its unit group. According to \cite{Gilmer} there is no such ring.
\end{proof}

For convenience, the following proposition summarizes the situation when the characteristic of $R$ is a Fermat prime.

\begin{prop} \label{charq}
Let $q$ be a Fermat prime. There is a ring $R$ of characteristic $q$ such that $R^{\times} \cong C_{p^{n}}$ if and only if 
$p = 2$ and $q = 2^n + 1 < \infty$, or $p = 2$, $q = 3$ and $n = 3$. 
\end{prop}
\begin{proof} This follows immediately from Proposition \ref{rel-prime}.\end{proof}

\begin{prop} \label{char2q}
Let $q$ be an odd prime. There is a ring $R$ of characteristic $2q$ such that $R^{\times} \cong C_{2^{n}}$ if and only if there is a ring $S$ of characteristic $q$ such that $S^{\times} \cong C_{2^{n}}$.\end{prop}
\begin{proof}
Suppose we have a ring $R$ with characteristic $2q$ such that $ R^{\times} \cong C_{2^{n}}$ for some $1 \le n \le \infty$. Consider the (two sided) ideals $(2)$ and $(q)$ in $R$. Note that since $q$ is odd, we have $(2) + (q) = (1)$. Moreover, $(2) \cap (q) = (0)$.  By the Chinese remainder theorem,
\[ R \cong R/(2) \times R/(q).\]
Hence, $R^\times  \cong ( R/(2) )^\times \times ( R/(q) )^\times$.  The group $R^{\times}$ is isomorphic to  $C_{2^{n}}$, an indecomposable abelian group. Therefore either $( R/(2) )^\times $ or $( R/(q) )^\times$ has to be the trivial group. But since $R/(q)$ is a ring of odd characteristic it must have non-trivial units. Thus $R/(q)$ is a ring of characteristic $q$ whose unit group is $C_{2^n}$.

Conversely, suppose we have a ring $R$ of characteristic $q$ such that $R^{\times} \cong C_{2^{n}}$. The ring $R \times \mathbf{F}_2$ is a ring of characteristic $2q$  (since $q$ is odd) whose unit group is also $C_{2^n}$.
\end{proof}

The foregoing analysis may now be used to prove the main theorem stated in \S \ref{sec:introduction}, following the proof outlined there. The next theorem provides a summary that ignores the connection between the unit group and the characteristic of $R$. For $n < \infty$, this theorem can also be derived using the general classification of rings whose unit group is a finite cyclic group in \cite{PS}. 

\begin{thm} \label{summary}
There is a ring $R$ such that $R^{\times} \cong C_{p^{n}}$ for some prime $p$ and $1 \le n \leq \infty$ if and only if $n < \infty$ and
\[p^{n} = \begin{cases} 
 8  &  \\
q - 1 &  \text{where }  q \text{ is a Fermat prime, or  } \\
p &  \text{where }  p \text{ is a Mersenne prime. } 
\end{cases}
\]
In particular, there is no ring $R$ whose unit group is a quasi cyclic $p$-group.
\end{thm}

\begin{proof}
It is clear that if $p^{n}$ is equal to one of the listed numbers, then there is a ring (in fact, a finite field) whose group of units is $C_{p^{n}}$. For the converse, observe that $C_{p^\infty}$ is never the multiplicative group of a ring (review Propositions \ref{char}, \ref{char0}, \ref{char2}, \ref{char4}, \ref{charq}, and \ref{char2q}), so we may assume $n < \infty$.

Now suppose that $p$ is odd. By Proposition \ref{char}, the characteristic of $R$ has to be 2. Proposition \ref{char2} implies that $n=1$ and $p$ is a Mersenne prime.
Next consider $p=2$. Proposition \ref{char} implies that the characteristic of $R$ has to be $0, 2, 4, q$ or $2q$ where $q$ is a Fermat prime. Since the cyclic groups $C_{2}, C_{2^{2}}$, and $C_{2^{3}}$  are each realizable as the multiplicative group of a finite field, we can assume that $n \ge 4$. 
(Note that $2^{3}$ is the only number in this list which is not of the form $q - 1$, where $q$ is a Fermat prime.) Then we can invoke Propositions \ref{char0}, \ref{char2}, and \ref{char4} to conclude that is no ring $R$ such that $R^{\times} \cong C_{p^{n}}$ when the characteristic of $R$ is $0$, $2$, or $4$, respectively.
In view of Proposition \ref{char2q}, it is enough to consider the case when the characteristic of $R$ is an odd prime $q$. In this case, $q$ must be a Fermat prime with $2^n = q - 1$ by Propositions \ref{charq}.
 \end{proof}

\begin{remark}
Theorem \ref{summary} also answers the following natural question. {\em Is the class of realizable abelian groups closed under specialization? } In other words, if $G$ is an abelian group that is the group of units of a commutative ring, is is true that every subgroup of $G$ is also the unit group of some commutative ring? The answer is no, as shown by the following example. Consider the group $C_{256}$.  This is a realizable group. In fact, since $257$ is a prime, we have  $\mathbf{F}_{257}^{\times} \cong C_{256}$. But  $C_{128}$, the unique subgroup of $C_{256}$ with index $2$, is not realizable by Theorem \ref{summary} because neither 128 nor 129 is a prime. In fact, it follows from  Proposition \ref{charq} that  $C_{256}$  is the smallest realizable indecomposable abelian group that is not closed under specialization.  The smallest decomposable example is $C_{10} \cong \mathbf{F}_{11}^\times$ because the subgroup $C_5$ is not realizable.
\end{remark}

\section{Torsion-free groups} \label{sec:torsion-free}
In this section we will show that every torsion-free abelian group $G$ is the group of units of some commutative ring.  A key ingredient in the  proof is the existence of a linear order on any torsion-free abelian group. A linear order on an abelian group $G$ is a binary relation $(\le)$ on $G$ that is reflexive, anti-symmetric, transitive, total, and satisfies translational invariance: if $a \le b$ in $G$, then for any $c$ in $G$, $ac \le bc$. An abelian group with a linear order is a partially ordered set in which any two elements are related by $\le$; such partially ordered sets are also called chains. A quintessential example of a linearly ordered abelian group is the group $(\mathbb{Q}, +)$, the additive group of rational numbers under the usual $\le$ relation.   In any linearly ordered abelian group, we write $a < b$, when $a \le b$ and $a \ne b$.

\begin{prop} Let $G$ be a linearly ordered abelian group.  If $a < b$ and $c \le d$, then $ac < bd$.  
\end{prop}
\begin{proof} Since $a < b$, we also have $a \le b$. Translational invariance implies that  $ac \le bc$.  But $ac \ne bc$, for otherwise $a = b$ by cancellation. Therefore $ac < bc$.  Since $c \le d$, translational invariance gives $cb \le bd$. Thus $ac < bc \le bd$.  Applying the transitive property, we have $ac \le bd$.  Again, $ac \ne bd$, because otherwise $ac \le bc \le ac$. By anti-symmetry, $ac = bc$, which is again a contradiction. So $ac < bd$ and this completes the proof.
\end{proof}

\begin{thm}[\cite{Levi}]\label{levi}
An abelian group $G$ admits a linear order if and only if it is torsion-free.
\end{thm}

We only need the `if' part of Levi's theorem. 

\begin{thm} \label{torsion-free}
If $G$ is a torsion-free abelian group, then $(\mathbf{F}_{2}[G])^{\times} \cong G$. In particular, indecomposable torsion-free abelian groups are unit groups of commutative rings.
\end{thm}

\begin{proof}
Recall that the elements in $G$ form an $\mathbf{F}_{2}$-basis for the group algebra $\mathbf{F}_{2}[G]$. Since $G$ is torsion-free, $G$ admits a linear order which we denote by $\le$.  Let $u$ be a unit in $\mathbf{F}_{2}[G]$. We can write 
\[ u = g_{1} + g_{2} + \cdots + g_{n},\]
for some $n \geq 1$ and  $g_{1} < g_{2} < \cdots < g_{n}$. 
We claim that $u$ has only one term; i.e., $n=1$. Suppose to the contrary $n > 1$. Then $g_{1} < g_{n}$. Let $v$ be the inverse of $u$. We write $v$ as
\[ v = h_{1} + h_{2} + \cdots + h_{m},\]
for some $m \geq 1$ and $h_{1} < h_{2} < \cdots < h_{m}$.
Since $uv = 1$, we have 
\[ \sum_{i=1}^{n} \sum_{j=1}^{m} g_{i} h_{j} = 1.\]
This can be written as 
\[ g_{1}h_{1} + g_{n}h_{m} + \text{ other terms } = 1.\]
Since $g_{1} < g_{n}$, and $h_{1} \le h_{m}$, the above proposition implies that  $g_{1}h_{1} < g_{n}h_{m}$,  $g_{1}h_{1}$ is the unique smallest term, and $g_{n}h_{m}$ is the unique largest term in the above summation. In particular, this shows that the elements of $G$ are not linearly independent, a contradiction.
\end{proof}

\begin{remark}
The second statement of Theorem \ref{torsion-free} provides an interesting contrast with a result in \cite{cl-1} that an indecomposable torsion-free abelian group of finite rank cannot be the multiplicative group of any  field. 
\end{remark}

%\begin{comment}
\bibliographystyle{alpha}

%\end{comment}
%\bibliography{rings}
%\bibliographystyle{alpha}

\end{document}